\documentclass[12pt,leqno]{amsart}

\usepackage{comment,color,bm,ulem,mathtools,empheq,amsmath,amssymb,amsfonts,stmaryrd,cite}
\usepackage[abbrev]{amsrefs}
\usepackage[top=30mm, bottom=30mm, left=25mm, right=25mm]{geometry}
\usepackage[pagewise]{lineno}

\AtBeginDocument{%
   \def\MR#1{}
} 
\mathtoolsset{showonlyrefs} 

\usepackage{amsthm}
\newtheorem{Theorem}{Theorem}[section]
\newtheorem{Lemma}[Theorem]{Lemma}

\newtheorem{Proposition}[Theorem]{Proposition}

\theoremstyle{definition}

\newtheorem{Remark}[Theorem]{Remark}

\numberwithin{equation}{section}

\newcommand{\R}{\mathbb R}

\newcommand{\Z}{\mathbb Z}

\newcommand{\Proj}{\mathbb P}
\newcommand{\F}{\mathcal F}

\usepackage{amsaddr}

\makeatletter
\renewcommand{\email}[2][]{%
  \ifx\emails\@empty\relax\else{\g@addto@macro\emails{,\space}}\fi%
  \@ifnotempty{#1}{\g@addto@macro\emails{\textrm{(#1)}\space}}%
  \g@addto@macro\emails{#2}%
}
\makeatother

\begin{document}
\title[]{Global well-posedness of the Navier--Stokes equations for small initial data in frequency localized Koch--Tataru's space}

\author[]{Alexey Cheskidov}
\address[Alexey Ceskidov]{Institute for Theoretical Sciences, Westlake University, No. 600 Dunyu Road Sandun Town, Xihu District, Hangzhou, Zhejiang, China}
\email{cheskidov@westlake.edu.cn}
\thanks{}

\author[]{Taichi Eguchi}
\address[Taichi Eguchi]{Department of Mathematics, Faculty of Science and Engineering, Waseda University, 3-4-1 Okubo, Shinjuku-ku, Tokyo, 169-8555, Japan}
\email{hasegawa-t@akane.waseda.jp}
\thanks{}



\date{}

\maketitle
\par \noindent
{\bf Funding}: 
The second author was supported by JST SPRING, [Grant Number JPMJSP2128]. 
\begin{abstract}
We construct global smooth solutions to the incompressible Navier--Stokes equations in $\R^3$ for initial data in $L^2$ satisfying some smallness condition. The high-frequency part is assumed to be small in $BMO^{-1}$, while the low-frequency part is assumed to be small only in $\dot B^{-1}_{\infty,\infty}$. Since $BMO^{-1}$ is strictly embedded in $\dot B^{-1}_{\infty,\infty}$, our assumption is weaker than that of Koch and Tataru (2001), which we also demonstrate with an example of finite energy divergence-free initial data. Also, our solutions attain the initial data in the strong $L^2$ sense, and hence satisfy the energy balance for all time.

\end{abstract}
\section{Introduction}
We consider the incompressible Navier--Stokes equations in $\R^3$:   
\begin{equation}
\label{NS}
\left\{
\begin{array}{ll}
\partial_t u -\Delta u + ( u \cdot \nabla ) u + \nabla p = 0 
& \text{in} \ \R^3 \times (0,T),\\
{\rm div}\, u = 0 
& \text{in}\ \R^3 \times (0,T),\\
u(0) = a 
& \text{in} \ \R^3, 
\end{array} 
\right.
\end{equation}
where $u=u(x,t)=(u_1(x,t), u_2(x,t), u_1(x,t))$ and $p=p(x,t)$ denote the vector field of the fluid and its pressure, 
while $a=a(x)=(a_1(x),a_2(x),a_3(x))$ is a given initial data. 
Our aim is to show the existence of a global-in-time smooth solution to \eqref{NS} for a small initial data in the framework of scaling invariant spaces. 
\par

\subsection{Scaling invariant spaces.}
Pioneer work on the local and global existence in scaling invariant spaces was done by Kato and Fujita: 
\begin{Theorem}[\cite{FujitaKato, KatoFujita}] \label{tham:KF}
There exists $\mu_0>0$ such that for any divergence-free $a \in \dot H^\frac{1}{2}(\R^3)$ satisfying  $\| a \|_{\dot H^\frac{1}{2}(\R^3)} < \mu_0$ there is a global smooth solution to \eqref{NS}. 
\end{Theorem}
The space $\dot H^{1/2} (\R^3)$ belongs to a class of scaling invariant spaces $X$ with the property that
\begin{align}
\| a \|_X = \| a_\lambda \|_X 
\,\,\,\,
\text{for all}
\,\,\,\,
0<\lambda, 
\end{align}
where $a_\lambda (x) := \lambda a( \lambda x)$ for $x \in \R^3$ and $\lambda>0$. 
More generally, a Bochner space $Y(0,\infty;X)$ is scaling invariant if  
\begin{align}
\| u \|_{Y(0,\infty;X)} = \| u_\lambda \|_{Y(0,\infty;X)} 
\,\,\,\,
\text{for all}
\,\,\,\,
0<\lambda, 
\end{align}
where $u_\lambda(x,t) := \lambda u (\lambda x, \lambda^2 t)$ for $(x,t) \in \R^3 \times (0,\infty)$ and $\lambda>0$. 
Such scaling admits the following property: 
\begin{align}
u\text{ is a solution of \eqref{NS} in } \R^3 \times (0,\infty) \Leftrightarrow 
u_\lambda\text{ is a solution of \eqref{NS} in } \R^3 \times (0,\infty).  
\end{align}
In general, it is anticipated that initial data $a$ in a scaling invariant (or critical) space $X$ gives rise to a solution in $L^\infty (0,T;X)$, with $T=\infty$ provided that $\|a\|_X$ is small enough. A natural problem is to find the largest critical space $X$ with such a property.
 
Following \cite{FujitaKato,KatoFujita}, Kato \cite{Kato} and Giga \cite{Giga} extended such a class of initial data to the scaling invariant Lebesgue space $L^3(\R^3)$. Kozono and Yamazaki \cite{KozonoYamazaki}, Cannone \cite{Cannone, Cannone2}, Planchon \cite{Planchon}, and Cannone and Planchon \cite{CannonePlanchon} reached scaling invariant Besov and Besov--Morrey spaces which are larger than $L^3(\R^3)$. For instance, they were able to reach  Besov spaces $\dot B^{-1+3/p}_{p,\infty} (\R^3)$ with $3\leqq p<\infty$. Finally,
Koch and Tataru \cite{KochTataru} proved global existence of mild solutions for small initial data in $BMO^{-1}(\R^3)$. 
Relationships between the above critical spaces, where the local or global well-posedness for small initial data holds, can be summarized as follows: 
\begin{align}
\dot H^{1/2}(\R^3) \subset L^3(\R^3) \subset \dot B^{-1+3/p}_{p,\infty} (\R^3) \subset BMO^{-1}(\R^3), 
\end{align}
where $3<p<\infty$. 
For further results on the global well-posedness, see 
Giga and Miyakawa \cite{GigaMiyakawa}, 
Lei and Lin \cite{LeiLin}, 
Iwabuchi and Nakamura \cite{IwabuchiNakamura}, 
and Kozono, Okada and Shimizu \cite{KozonoOkadaShimizu}. 
In frequency localized Besov and $BMO^{-1}(\R^3)$ spaces, for local-in-time existence theorems, we also refer \cite{IwabuchiNakamura} and \cite{KochTataru}. 

Since their construction in \cite{KochTataru}, regularity properties of Koch-Tataru solutions have been extensively studied. Analyticity in space was proved by Miura and Sawada \cite{MiuraSawada}, as well as Germain, Pavlovi\'c, and Staffilani \cite{GermainPavlovicStaffilani} for all positive time. The best-known temporal regularity, recently proved by Hou in \cite{2410.16468}, states that Koch-Tataru solutions are weak*-continuous in $BMO^{-1}$.

In contrast to global well-posedness, Bourgain and Pavlovi\`{c} \cite{BourgainPavlovic} proved that \eqref{NS} is ill-posed in the largest critical space $\dot B^{-1}_{\infty,\infty}(\R^3)$ in the sense of norm inflation. It should be noted that the following embedding holds: 
\begin{align}
BMO^{-1}(\R^3) \subset \dot B^{-1}_{\infty,\infty}(\R^3). 
\end{align} 
For further results, we also refer to, e.g., Yoneda \cite{Yoneda}, Wang \cite{Wang}, and Iwabuchi and Takada \cite{IwabuchiTakada}. 
In the above ill-posedness results, the initial data is constructed so that it is supported on large frequencies, but nonlinear interactions produce fast growth on low frequencies. 

In this paper, we study solutions to the 3D NSE with finite energy divergence-free initial data which is small in the class
\begin{equation} \label{eq:space}
X=BMO_{\sqrt{\delta}}^{-1}(\R^3) \cap \dot B^{-1}_{\infty,\infty}(\R^3),
\end{equation}
where $0<\sqrt{\delta}<<1$ is a small length scale. 
Here the function space $BMO^{-1}_{\sqrt{\delta}}(\R^3)$ denotes a local $BMO^{-1}(\R^3)$ space in which only balls of size $\sqrt{\delta}$ and smaller are considered, i.e., roughly speaking, the part of the initial data above frequency $1/\sqrt{\delta}$. In other words, only the high-frequency part of the initial data is required to be small in $BMO^{-1}(\R^3)$, while the rest is assumed to be small in a larger space $\dot B^{-1}_{\infty,\infty}(\R^3)$. The precise value of $\delta$ is specified in the main result below.

For such initial data, we construct a global smooth solution to \eqref{NS} in $\R^3 \times (0,\infty)$, which is also continuous in $L^2(\R^3)$ at the initial time. 
Since the smallness condition in $\dot B^{-1}_{\infty,\infty}$ is weaker than that in $BMO^{-1}$, our assumption is weaker than that of \cite{KochTataru}, which we also demonstrate with an example of divergence-free initial data with arbitrary large energy. More precisely, we construct a vector field with arbitrary large (but finite) energy, which is arbitrary small in $BMO^{-1}_{\sqrt{\delta}}(\R^3) \cap \dot B^{-1}_{\infty,\infty}(\R^3)$, but arbitrary large in $BMO^{-1}(\R^3)$. 

We emphasize that our constructed solution attains the initial data strongly in $L^2$:
\begin{equation} \label{eq:L^2-convergence}
u(t) \to a \qquad \text{as} \qquad t \to 0+,
\end{equation}
in $L^2(\R^2)$. Recall that the best known temporal regularity of Koch-Taturu solutions is the weak* continuity in $BMO^{-1}$  \cite{2410.16468}. So by \eqref{eq:L^2-convergence} we rule out a possible energy jump at the initial time.

\subsection{Energy balance.}
In 1934, in the pioneering work \cite{Leray}, for any finite energy divergence-free initial data, Leray proved global existence of weak solutions to the 3D NSE satisfying the following energy inequality:
\begin{equation} \label{eq:energy-in}
\frac12 \|u(t)\|_{L_x^2}^2 \leq \frac12 \|u(t_0)\|_{L_x^2}^2 -  \int_{t_0}^t \|\nabla u(\tau)\|_{L_x^2}^2 \, d\tau,
\end{equation}
for all $t \geq t_0$, a.a. $t_0\geq 0$ (including $t_0=0$). Such weak solutions are usually referred to as Leray-Hopf solutions. In general, the existence of Leray-Hopf solutions satisfying the energy equality is not known. In fact, even the existence of Leray-Hopf solutions with continuous or decreasing energy are open questions. In \cite{2407.17463}, the existence of weak solutions with continuous energy was shown for any divergence-free finite energy initial data, but those solutions are not Leray-Hopf.

In this paper, for any divergence-free initial data which is small in space \eqref{eq:space}, we construct a global solution $u(t)$ which is smooth for positive time, but also continuous in $L^2$ at the initial time, so 
\begin{equation} \label{eq:cont_u_in_L^2}
u \in C([0,\infty);L^2(\R^3)).
\end{equation}
Moreover, since the constructed solution $u(t)$ is smooth for positive time, it satisfies the energy balance for positive time. Combined with \eqref{eq:cont_u_in_L^2}, this implies the energy balance starting from the initial time:
\[
\frac12 \|u(t)\|_{L_x^2}^2 = \frac12 \|u(0)\|_{L_x^2}^2 -  \int_{0}^t \|\nabla u(\tau)\|_{L_x^2}^2 \, d\tau,
\]
for all $t>0$. Therefore $u(t)$ is a Leray-Hopf solution satisfying the energy equality.

\par
This paper is organized as follows: in Section $2$, we state our main result. 
Section $3$ is devoted to some preliminaries. In Section $4$, we prove our main result. In Section $5$, we construct an example of initial data satisfying the assumptions of our main theorem, but with arbitrary large $BMO^{-1}$ norm.
\section{Main Result}
Define a function $\delta: \R_+ \times \R_+ \to \R_+$ given by 
\begin{equation} 
(\eta, E) \mapsto
\delta(\eta, E) \equiv \mu_0^{-2} E e^{-\frac{1}{C_0 \eta}}, 
\end{equation}
where $C_0$ is an absolute constant determined later and $\mu_0$ is the Kato-Fujita constant in Theorem~\ref{tham:KF}. 
We also define $T_* = T_* (a)$ for $a \in L^2(\R^3)$ as 
\begin{align}
T_*=T_*(a) = \mu_0^{-2} \| a \|_{L^2(\R^3)}^2. 
\end{align} 
\begin{Theorem} \label{thm1}
There exists $\varepsilon_0>0$, 
such that for any $\varepsilon <\varepsilon_0$ and divergence free $a \in L^2(\R^3)$ satisfying
\begin{align}
\| a \|_{\dot B^{-1}_{\infty,\infty}(\R^3)} < \varepsilon, \qquad
\| a \|_{BMO^{-1}_{\sqrt{\delta}}(\R^3)} <  \varepsilon, \qquad \delta = \delta(\varepsilon,\| a \|_{L^2(\R^3)}^2),
\end{align} 
there exists a global smooth solution $u \in C^\infty(\mathbb{R}^3 \times (0, \infty))$ to \eqref{NS} such that
\[
u(t) \to a \qquad \text{as} \quad t\to 0+
\]
in $L^2(\R^3)$.

In particular, $u(t)$ satisfies the energy equality
\[
\frac12 \|u(t)\|_{L_x^2}^2 = \frac12 \|u(0)\|_{L_x^2}^2 -  \int_{0}^t \|\nabla u(\tau)\|_{L_x^2}^2 \, d\tau,
\]
for all $t>0$.
\end{Theorem}
\begin{Remark}
The parameter $\delta$ depends on $\varepsilon$ and converges very rapidly to $0$ as $\varepsilon$ approches $0$. 
This means that smallness of the $BMO^{-1}$ norm is necessary only at very high frequencies. 
Since we do not require smallness of the low-frequency part in $BMO^{-1}(\R^3)$, our assumption is weaker than that in \cite{KochTataru}, which we further illustrate with an example of finite energy divergence-free initial data in Theorem~\ref{Thm:example}.
\end{Remark}
We note that our argument yields a stronger result below. However, the main novelty lies in the smallness of the length scale $\sqrt{\delta}$ compared to $T_*$, and thus we focus on Theorem~\ref{thm1}.
\begin{Theorem} \label{thm2}
There exists $\varepsilon_0>0$, 
such that for any $\varepsilon <\varepsilon_0$ and divergence free $a \in L^2(\R^3)$ satisfying
\begin{align}
\| P_{\geq\frac{1}{\sqrt{T_*(a)}}} a \|_{\dot B^{-1}_{\infty,\infty}(\R^3)} <\varepsilon, \qquad
\| P_{\geq\frac{1}{\sqrt{\delta}}} a \|_{BMO^{-1}(\R^3)} < \varepsilon, \qquad \delta = \delta(\varepsilon,\| a \|_{L^2(\R^3)}^2),
\end{align} 
there exists a global smooth solution $u \in C^\infty(\mathbb{R}^3 \times (0, \infty))$ to \eqref{NS} such that
\[
u(t) \to a \qquad \text{as} \quad t\to 0+
\]
in $L^2(\R^3)$.
\end{Theorem}
In this theorem we used $P_{\geq \lambda}$ for a projection onto frequencies above $\lambda$.

\section{Preliminaries}
\subsection{Notations}
For a function $f: \R^3 \to \R^3$, we define 
\begin{align}
\| f \|_p
= 
\| f \|_{L^p(\R^3)} 
= \bigg \{ \int_{\R^3} | f(x) |^p \, dx \bigg \}^{\frac{1}{p}}. 
\end{align}
The Besov spaces $\dot B^{s}_{p,r} = \dot B^{s}_{p,r}(\R^3)$ and the Triebel--Lizorkin spaces $\dot F^{s}_{p,r} = \dot F^{s}_{p,r}(\R^3)$ are defined via the Littlewood--Paley decomposition $f=\sum_{q\in \Z} f_q$. 
Let $\phi\in C^\infty_0(\R^3)$ with $\text{supp} \, \phi \subset \{ \xi=(\xi_1,\xi_2,\xi_3); \,\, 1/2 < |\xi| < 2 \}$ satisfying $|\phi_q(\xi)| \leq 1$ and 
\begin{align}
\sum_{q \in \Z} \phi_q(\xi) \equiv 1 \qquad \xi \in \R^3, 
\end{align}
where $\phi_q (\xi) := \phi(\lambda_q \xi)$, $\lambda_q := 2^q$, $q \in \Z$, $\xi \in \R^3$. 
We use $f_q (x) := \F^{-1}[\phi_q(\xi) \hat f(\xi) ] (x) $, 
$f_{< q}(x) := \sum_{q^\prime < q} f_{q^\prime}(x)$ for $x \in \R^3$.  
For $s\in \R$, $1\leqq p,r \leqq \infty$, we define Besov norms
\begin{align}
\| f \|_{\dot B^s_{p,r}} = \| f \|_{\dot B^s_{p,r}(\R^3)} := 
\left \{ 
\begin{array}{ll}
\Bigg \{ \displaystyle \sum_{q \in \Z} (\lambda_q^s \| f_q \|_p )^r \Bigg \}^{\frac{1}{r}} & 1 \leqq r < \infty, \\
\displaystyle \sup_{q \in \Z} \lambda_q^s \| f_q\|_p & r = \infty. 
\end{array}
\right. 
\end{align}
As for the Triebel--Lizorkin spaces, for $1\leqq p < \infty$, we define 
\begin{align}
\| f \|_{\dot F^s_{p,r}} = \| f \|_{\dot F^s_{p,r}(\R^3)} := 
\left \{ 
\begin{array}{ll}
\Bigg \|
\bigg \{ \displaystyle \sum_{j \in \Z} (\lambda_q^s | f_q | )^r \bigg \}^{\frac{1}{r}} \Bigg \|_p & 1 \leqq r < \infty, \\
\Big \| \displaystyle \sup_{q \in \Z} \lambda_q^s | f_q | \Big \|_p & r = \infty. 
\end{array}
\right.
\end{align}
For $p=\infty$, we define 
\begin{align}
\| f \|_{\dot F^{s}_{\infty,r}} = 
\| f \|_{\dot F^{s}_{\infty,r}(\R^3)} := 
\bigg \{ 
\sup_{Q_{z,j} \in \mathcal{Q}} 
\frac{1}{|Q_{z,j}|} \int_{Q_{z,j}} \sum_{q=j}^\infty (\lambda_q^s |f_q(x)|)^r \, dx
\bigg \}^\frac{1}{r},
\end{align}
where $\mathcal{Q} \equiv \{ Q_{z,j}; (x_1,x_2,x_3) \in Q_{z,j}, \,\, 
\lambda_j^{-1} z_i \leqq x_i \leqq \lambda_j^{-1} (z_i + 1), \,\,
j \in \Z, \,\, z=(z_1,z_2,z_3) \in \Z^3 \}$. 
We also use the equivalent norms: 
\begin{align}
\| f \|_{\dot B^{-1}_{\infty,\infty}} \sim \sup_{0<t} t^\frac{1}{2} \| e^{t \Delta} f \|_\infty, 
\end{align}
\begin{align}
\| f \|_{BMO^{-1}} \sim \| f \|_{\dot F^{-1}_{\infty,2}} \sim 
\sup_{x\in \R^3, \,\, 0<t} \bigg \{ \frac{1}{|B(x,\sqrt{t})|}  \int_{B(x,\sqrt{t})} \int^{t}_0 |e^{s\Delta} f(y)|^2 \,ds dy\bigg \}^\frac{1}{2}, 
\end{align}
\begin{align}
\| f \|_{BMO^{-1}_{\sqrt{\delta}}} \sim 
\sup_{x\in \R^3, \,\, 0<t<\delta} \bigg \{ \frac{1}{|B(x,\sqrt{t})|}  \int_{B(x,\sqrt{t})} \int^{t}_0 |e^{s\Delta} f(y)|^2 \,ds dy\bigg \}^\frac{1}{2}. 
\end{align}
For more detail on the properties of the function spaces, see Bahouri, Chemin and Danchin \cite{BCD}. 
\subsection{Settings}
First we remark that for $\varepsilon$ sufficiently smaller than $C_0^{-1}$, we have 
\begin{align}
\delta = \delta(\varepsilon, \| a \|_2^2) 
= 
\mu_0^{-2} \| a \|_2^2 e^{-\frac{1}{C_0 \varepsilon}}
= 
T_* e^{-\frac{1}{C_0 \varepsilon}} 
< T_*.
\end{align}
Now for $0<\delta<T_*$, we now introduce a Banach space $X=X_{T_*,\delta}$ with the norm: 
\begin{align}
\| u \|_X := \max ( \| u \|_{0,T_*}, \llbracket u \rrbracket_{\delta} ), 
\end{align} 
where 
\begin{align}
\|u\|_{0,T_*} := \sup_{0<t<T_*} t^\frac{1}{2} \| u(t) \|_\infty, 
\end{align}
and
\begin{align}
\llbracket u \rrbracket_\delta := \sup_{x\in \R^3, \,\, 0<t<\delta} \bigg \{ \frac{1}{|B(x,\sqrt{t})|}  \int_{B(x,\sqrt{t})} \int^{t}_0 |u(y,s)|^2 \,ds dy\bigg \}^\frac{1}{2}. 
\end{align}
Throughout the proof we will also use  
\begin{align}
\|u\|_{0,\delta} := \sup_{0<t<\delta} t^\frac{1}{2} \| u(t) \|_\infty, 
\qquad
\|u\|_{\delta,T_*} := \sup_{\delta<t<T_*} t^\frac{1}{2} \| u(t) \|_\infty. 
\end{align}
\subsection{Nonlinear Estimates} For functions $u,v: \R^3\times(0,T) \to \R^3$, we define 
\begin{equation}
N(u,v)(x,t) = \int^t_0 e^{(t-s)\Delta} \Proj \nabla \cdot (u\otimes v) (x,s) \, ds 
= \int^t_0 \int_{\R^3} K(t-s,x-y) (u\otimes v) (y,s) \, dy ds. 
\end{equation}
For the kernel $K$, we have the following point-wise estimate: 
\begin{Proposition}[\cite{KochTataru}] \label{kernel}
For $x\in \R^3$ and $t>0$, we have 
\begin{align}
|K(x,t)| \leqq C_1 (\sqrt{t} + |x|)^{-4}, 
\end{align}
where $C_1$ is an absolute constant. 
\end{Proposition} 
Now, we prove the key estimate for the nonlinear term.
\begin{Lemma} \label{nonlinear estimate 1}
For $u,v \in X$, we have 
\begin{align} \label{NLE}
\| N(u,v) \|_{\delta,T_*} 
\leqq &\,
C (\| u \|_{0,\delta} + \llbracket u \rrbracket_{\delta}) 
(\| v \|_{0,\delta} + \llbracket v \rrbracket_{\delta}) \\
&+ 
C_1  \ln \bigg ( \frac{1+\sqrt{1-\delta/T_*}}{1-\sqrt{1-\delta/T_*}} \bigg ) 
\|u\|_{\delta,T_*}\|v\|_{\delta,T_*}, 
\end{align}
where $C$ is an absolute constant. 
\end{Lemma}
\begin{Remark}
One of the differences between \eqref{NLE} and the approach in \cite{KochTataru} is a decomposition into frequencies below and above $\frac{1}{\sqrt{\delta}}$. We also introduce a new quantity (the logarithmic part in the R.H.S of \eqref{NLE}) coming from the incomplete Beta function to estimate $N(u,u)$ by the frequency localized $BMO^{-1}$ norm.
Moreover, on the time interval $\delta<t<T_*$, we cover the  spacial region by smaller cubes than in \cite{KochTataru} to avoid a singularity coming from $\delta \to 0$. 
\end{Remark}
\begin{proof}
Let $\delta<t<T_*$, $x \in \R^3$. 
We split $N(u,v)(x,t)$ into 
\begin{align}
I_1 &= \int^\delta_{\delta/2} \int_{\R^3} K(t-s,x-y) (u\otimes v)(y,s)\, dyds.\\
I_2 &= \int^{\delta/2}_0 \int_{\R^3} K(t-s,x-y) (u\otimes v)(y,s)\, dyds,\\
I_3 &= \int^t_\delta \int_{\R^3} K(t-s,x-y) (u\otimes v)(y,s)\, dyds. 
\end{align}
For $I_1$, by Proposition \ref{kernel}, we have 
\begin{align}
t^\frac{1}{2}|I_1| 
\leqq 
C_1 
t^\frac{1}{2} \int^\delta_{\delta/2} s^{-1} (t-s)^{-\frac{1}{2}} \, ds \, 
\| u \|_{0,\delta}  \| v \|_{0,\delta}\leqq C_1 \ln(3+2\sqrt{2}) \| u \|_{0,\delta}  \| v \|_{0,\delta}. 
\end{align}
For $I_2$, we also split into 
\begin{align}
I_{2,1} &= \int^{\delta/2}_0 \int_{Q(x,\sqrt{t})} K(t-s,x-y) (u\otimes v)(y,s)\, dyds,\\
I_{2,2} &= \int^{\delta/2}_0 \int_{Q(x,\sqrt{t})^c} K(t-s,x-y) (u\otimes v)(y,s)\, dyds,
\end{align}
where $Q(x,\sqrt{t})$ is a cube with a side length of $2\sqrt{t}$ centered at $x$. For $I_2$, it follows from Proposition \ref{kernel} that 
\begin{align}
t^\frac{1}{2}|I_{2,1}| 
&\leqq C t^\frac{1}{2} \int^{\delta/2}_0 \int_{Q(x,\sqrt{t})}  \frac{|u(y,s)| |v(y,s)|}{\sqrt{t}^{4}} \,dy ds  \\
&\leqq \frac{C\sqrt{\delta}^3}{\sqrt{t}^{3}} 
\sum_{z \in A(x,\sqrt{t})} \frac{1}{|B(z,\sqrt{\delta})|} 
\int^{\delta/2}_0 
\int_{B(z,\sqrt{\delta})} |u(y,s)| |v(y,s)| \,dy ds  \\
&\leqq 
C
\llbracket u \rrbracket_\delta \llbracket v  \rrbracket_\delta, 
\end{align}
where $A(x,\sqrt{t}) \subset \mathbb{Z}^3
$ is such that $\{B(z,\sqrt{\delta})\}_{z \in A(x,\sqrt{t})}$ covers 
$Q(x,\sqrt{t})$ and $|A(x,\sqrt{t})| \sim (\sqrt{t}/\sqrt{\delta} )^3$. Here, $|A(x,\sqrt{t})|$ denotes the number of the elements of the set $A(x,\sqrt{t})$. 
For $I_{2,2}$, we have 
\begin{align}
t^\frac{1}{2} |I_{2,2}| 
\leqq 
C_1 t^\frac{1}{2} \sum_{n=2}^\infty \int^{\delta/2}_0 \int_{Q(x,(n+1)\sqrt{t}/2)\backslash Q(x,n\sqrt{t}/2)} \frac{|u(y,s)||v(y,s)|}{(n\sqrt{t})^{4}} \,dy ds. \\
\end{align}
Now, let $L(x,n\sqrt{t}/2)$ be the set of the lattice point on the surface of $Q(x,n\sqrt{t}/2)$ such that $\{Q(z,\sqrt{t})\}_{z \in L(x,n\sqrt{t}/2)}$ covers $Q(x,(n+1)\sqrt{t}/2)\backslash Q(x,n\sqrt{t}/2)$ and $| \mathcal{L}(x,n\sqrt{t}/2)| \sim n^2$. Then we have 
\begin{align}
t^\frac{1}{2} |I_{2,2}| 
&\leqq 
\frac{C}{\sqrt{t}^{3}} \sum_{n=2}^\infty 
\sum_{z \in L(x,n\sqrt{t}/2)} 
\frac{1}{n^4} \int^{\delta/2}_0 
\int_{{Q(z,\sqrt{t})}} |u(y,s)||v(y,s)| \,dy ds\\
&\leqq 
\frac{C \sqrt{\delta}^3}{\sqrt{t}^{3}} \sum_{n=2}^\infty 
\sum_{z \in L(x,n\sqrt{t}/2)} 
\sum_{w \in A(z,\sqrt{t})} 
\frac{1}{n^4 |B(w,\sqrt{\delta})|} \int^{\delta/2}_0 
\int_{{B(w,\sqrt{\delta})}} |u(y,s)||v(y,s)| \,dy ds \\
&\leqq 
\frac{C \sqrt{\delta}^3}{\sqrt{t}^{3}}
\cdot 
\frac{\sqrt{t}^3 \llbracket u \rrbracket_\delta \llbracket v \rrbracket_\delta}{\sqrt{\delta}^3}
\sum_{n=2}^\infty  
\frac{1}{n^2}. 
\end{align}
For $I_3$, we have 
\begin{align}
t^\frac{1}{2}|I_3| 
&\leqq 
C_1 
\|u\|_{\delta,T_*}\|v\|_{\delta,T_*} t^\frac{1}{2} \int^t_\delta s^{-1} (t-s)^{-\frac{1}{2}} \, ds = 
C_1 F(\delta/t) \|u\|_{\delta,T_*}\|v\|_{\delta,T_*}, 
\end{align}
where the above incomplete beta function is equal to $F$ given by 
\begin{align}
F(\delta/t) \equiv \ln \bigg ( \frac{1+\sqrt{1-\delta/t}}{1-\sqrt{1-\delta/t}} \bigg ). 
\end{align}
Since $F(\tau)$ is monotonically decreasing on $(0,1]$, we have $F(\delta/t) \leqq F(\delta/T_*)$. 
This completes the proof of Lemma \ref{nonlinear estimate 1}. 
\end{proof}
\par 
Similarly, we the following estimates hold (see also Koch--Tataru \cite{KochTataru}). 
\begin{Lemma} \label{nonlinear estimate 2}
For $u,v \in X$, we have 
\begin{align}
\| N(u,v) \|_{0,\delta} \leqq 
C (\| u \|_{0,\delta} + \llbracket u \rrbracket_\delta)
(\| v \|_{0,\delta} + \llbracket v \rrbracket_\delta), 
\end{align}
\begin{align}
\llbracket  N(u,v) \rrbracket_\delta \leqq 
C (\| u \|_{0,\delta} + \llbracket u \rrbracket_\delta)
(\| v \|_{0,\delta} + \llbracket v \rrbracket_\delta),  
\end{align}
where $C$ is an absolute constant. 
\end{Lemma}
Next, we have the following estimate. 
\begin{Lemma} \label{energy estimate}
For $0 \leqq \alpha < 1$, we have 
\begin{align}
\sup_{0<t<T_*} t^\frac{\alpha}{2} \|\Lambda^\alpha N(u,u) \|_2 \leqq 
C_2(\alpha)
\| u \|_{0,T_*}
\sup_{0< t< T_*}
t^\frac{\alpha}{2}\| \Lambda^\alpha u(t)\|_2, 
\end{align}
where $C_2=C_2(\alpha)$ is a constant depending on $\alpha$. 
\end{Lemma}
\begin{proof}
By the Leibniz rule (see Kato--Ponce \cite{KatoPonce}), we have 
\begin{align}
\sup_{0<t<T_*} t^\frac{\alpha}{2} \| \Lambda^\alpha N(u,u) \|_2 
&\leqq 
C
\sup_{0<t<T_*} t^\frac{\alpha}{2} \int^t_{0} (t-s)^{-\frac{1}{2}} \| \Lambda^\alpha (u\otimes u)(s)\|_2 \, ds\\
&\leqq 
C(\alpha)
\sup_{0<t<T_*} t^\frac{\alpha}{2} \int^t_{0} (t-s)^{-\frac{1}{2}} \| \Lambda^\alpha u(s)\|_2 \| u(s) \|_\infty \, ds \\
&\leqq 
C(\alpha)
B\bigg ( \frac{1}{2}, \frac{1}{2} -\frac{\alpha}{2} \bigg )
\| u \|_{0,T_*}
\sup_{0 <t < T_*}
t^\frac{\alpha}{2}\| \Lambda^\alpha u(t)\|_2,
\\
\end{align}
where $B$ denotes the beta function. 
\end{proof}
\section{Proof of Theorem \ref{thm1}} 
We now introduce a map $\Phi$ as 
\begin{align}
\Phi u = e^{t\Delta} a - N(u,u). 
\end{align}
We also define a closed subset $S \subset X$ as 
\begin{align}
S \equiv \{ u \in X: 
\| u \|_X \leqq 4 \varepsilon \}. 
\end{align}
Since, as $x \to 0+$,
\[
\ln \left(  \frac{1+\sqrt{1-x}}{1-\sqrt{1-x}} \right) = -\ln x +O(1),
\]
recalling that $\delta/T_* = e^{-\frac{1}{C_0 \varepsilon}}$,
there exists an absolute constant $C_0 \ll C_1$, such that for any sufficiently small $\varepsilon>0$, it holds that 
\[
\begin{split} \label{adv}
\ln \bigg ( \frac{1+\sqrt{1-\delta/T_*}}{1-\sqrt{1-\delta/T_*}} \bigg ) &= \ln \left( \frac{1+\sqrt{1-e^{-\frac{1}{C_0\varepsilon}}}}{1-\sqrt{1-e^{-\frac{1}{C_0\varepsilon}}}} \right)\\
&< \frac{1}{16 C_1 \varepsilon}.
\end{split}
\]
Hence, 
by Lemma \ref{nonlinear estimate 1}, \ref{nonlinear estimate 2} and \eqref{adv}, we have 
\begin{align} \label{eq:mapping_estimate}
\| \Phi u \|_X \leqq 4 \varepsilon, \qquad
\| \Phi u - \Phi v \|_X \leqq \frac{3}{4} \| u - v \|_X,  
\end{align} 
for $u,v \in S$. 
This implies that $\Phi$ maps from $S$ to $S$, and it is a contraction mapping. Hence, a fixed point $u \in X$ exists with $\Phi u = u$. 
Moreover, by Lemma \ref{energy estimate}, 
for $0 \leqq \alpha < 1$ and $0<\tau\leqq T_*$, we have 
\begin{align} \label{energy}
(1 - \varepsilon C_2(\alpha)) \sup_{0 < t< T_*} t^\frac{\alpha}{2} \| \Lambda^\alpha u(t) \|_2 
\leqq 
\sup_{0<t<T_*} t^\frac{\alpha}{2} \|\Lambda^\alpha e^{t\Delta} a \|_2.
\end{align}
In particular, considering $\alpha=0$ in \eqref{energy}, we obtain $u \in L^\infty(0,T_*; L^2)$ and 
\begin{align} \label{energy jump}
\limsup_{t \to 0} \| u(t) \|_2^2 \leqq 2 \| a \|_2^2, 
\end{align}
for sufficiently small $\varepsilon>0$. 
In addition, taking $1/2 < \alpha < 1$, by a Sobolev embedding, we may choose $p$ and $q$ so that $3<p<\infty$, $2/q+3/p=1$, and 
\begin{align} \label{PS}
u \in L^q(\delta^\prime,T_*;L^p) 
\,\,\,\,
\text{for all} 
\,\,\,\,
0<\delta^\prime<\delta.
\end{align} 
Note that \eqref{PS} corresponds to the classical Ladyzhenskaya--Prodi--Serrin class. 
By Fabes--Jones--Riviere \cite[Theorem 2.1] {FabesJonesRiviere} (see Theorem \ref{FJR} in the Appendix), $u(t)$ is a weak solution on $[0,T_*]$. Moreover, it is well known that $u \in C^\infty(\mathbb{R}^3 \times (0, T^*))$  since it belongs to the Ladyzhenskaya--Prodi--Serrin class. In particular, it satisfies the energy equality
\begin{align} \label{ee}
\frac{1}{2} \| u(t) \|_2^2 +  \int^t_{t_0} \| \nabla u(s) \|_2^2 \, ds = \frac{1}{2} \| u(t_0) \|_2^2 
\,\,\,\,
\text{for all}
\,\,\,\,
0< t_0 \leqq t\leqq T_*.
\end{align}

Next we prove that $u(t)$ can be extended to a global smooth solution of the Navier-Stokes equations. Combining \eqref{energy jump} and \eqref{ee}, we have 
\begin{align} \label{H1}
\int^{T_*}_{0} \| \nabla u(s) \|_2^2 \, ds \leqq \| a \|_2^2.
\end{align}
Hence, there exists $0 < T_{**} < T_*$ such that 
\begin{align}
\| \nabla u(T_{**}) \|_{2} \leqq \mu_0.
\end{align}
Indeed, suppose to the contrary that for every $0 < T_{**} < T_*$, it holds that 
\begin{align}
\| \nabla u(T_{**}) \|_{2} > \mu_0.
\end{align}
Then, we have 
\begin{align}
\int^{T_*}_0 \| \nabla u (s) \|_{2}^2 \, ds > T_* \mu_0^2 = \| a \|_2^2,
\end{align}
a contradiction. By Theorem~\ref{tham:KF}, there exists a global smooth solution  $u' \in C^\infty (\mathbb{R}^3 \times (T_{**},\infty))$ to \eqref{NS} with $u'(T_{**})= u(T_{**})$. By strong-strong uniqueness, $u(t)=u'(t)$ on $[T_{**},T_*]$. Therefore, we can extend $u(t)$ to the whole time interval $[0, \infty)$ defining $u(t):=u'(t)$ for $t > T_{**}$.

\par
Next, we prove the continuity of $u(t)$ at $t=0$ in $L^2(\R^3)$. Since $u \in X$, we have 
\begin{align} \label{CL cond}
\Big \| \| u(t) \|_{L^\infty(\R^3)} \Big \|_{L^2_w(0,T_*)} \leqq 4\varepsilon \|t^{-\frac{1}{2}}\|_{L^2_w(0,T_*)} < \infty. 
\end{align} 
Hence, combining \eqref{H1} and \eqref{CL cond}, we have 
\begin{align} 
u \in L^2_w(0,T_*;L^\infty) \cap L^2(0,T_*;\dot H^1) \cap L^\infty(0,T_*;L^2).
\end{align} 
By a result of Cheskidov--Luo \cite[Theorem 1.3]{CheskidovLuoE}, 
it holds that 
\begin{align}
\frac{1}{2} \| u(t) \|_2^2 + \nu \int^t_0 \| \nabla u(s) \|_2^2 \, ds = \frac{1}{2} \| a \|_2^2 
\,\,\,\,
\text{for}
\,\,\,\,
0 \leqq t\leqq T_*.
\end{align}
This implies the continuity of $u(t)$ at $t=0$ in $L^2$. 
\section{Example of initial data with large $BMO^{-1}$ norm}
In this section, we construct a family of initial data to demonstrate that the assumption of Theorem \ref{thm1} is weaker compared to \cite{KochTataru}. Recall that $a_{>q_0}$ and $a_{\leqq q_0}$ denote Littlewood--Paley projections on frequencies above and bellow $\lambda_{q_0}:=2^{q_0}$ respectively.

\begin{Theorem} \label{Thm:example}
There is an absolute constant $C>0$ and a function $E_0(\varepsilon,M)$ with 
\[
\lim_{\varepsilon \to 0} E_0(\varepsilon,M) =0,
\]
when $M\leqq C$,
such that for any $\varepsilon>0$, $M>0$, $E>E_0(\varepsilon,M)$, and
\[
\delta = \mu_0^{-2} E e^{-\frac{1}{C_0 \varepsilon}},
\]
there exists divergence-free, real-valued $a \in L^2(\R^3)$ such that
\[
\|a\|_{L^2}^2=E, \qquad \|a \|_{\dot B^{-1}_{\infty,\infty}} < \varepsilon,  \qquad a_{>\frac{1}{\sqrt{\delta}}} =0,  \qquad  \text{but} \qquad \| a \|_{BMO^{-1}} > M.
\]
\end{Theorem}
This theorem provides initial data $a$ with arbitrary large energy satisfying the assumptions of Theorem \ref{thm1}. Therefore there exists global smooth solution of the 3D Navier-Stokes equations with such initial data. On the other hand, the $BMO^{-1}$ norm of $a$ can be arbitrary large.

\subsection{Settings}
Let $q_1$ be an integer satisfying $\lambda_{q_1+1}\leqq 1/\sqrt{\delta}$. We will construct
divergence-free, real-valued $a \in L^2(\R^3)$ such that 
the frequency support of $a$ is the inside of the ball $B(0,\lambda_{q_1+1}) = \{ \xi =(\xi_1,\xi_2,\xi_3); |\xi| \leqq \lambda_{q_1+1}\}$, i.e., 
$a_{>q_1+1} = 0$, and $\| a \|_{\dot B^{-1}_{\infty,\infty}} < \varepsilon$, but $\| a \|_{BMO^{-1}} > 1/\varepsilon$.

Let $q_0 \in \Z$ be a parameter determined later, 
and let $\sigma = \lambda_{q_0}$. 
Note that $\sigma \to 0$ as $q_0 \to -\infty$. 
We now define $Q_q$ and $Q_q^\prime$ as 
\[
Q_q := \{ \xi = (\xi_1,\xi_2,\xi_3); 
\,\,
|\xi_1 - \lambda_q |,
\,\,
|\xi_2|,
\,\,
|\xi_3| < \sigma/8
\}, 
\,\,\,\,
Q_q^\prime := - Q_q, 
\]
and then for $\sigma \leqq \lambda_q$, we have 
\begin{align}
\pm Q_q \subset \{ \xi = (\xi_1, \xi_2, \xi_3); \,\, | \xi_1 \mp \lambda_q |, \,\, |\xi_2|, \,\, |\xi_3| < \lambda_q/8  \} \subset \text{supp}\, \phi_q. 
\end{align} 
Then, we define 
\begin{align}
a(x) := \frac{\varepsilon}{2 C^\prime \sigma^3} \sum_{q = q_0}^{q_1} \lambda_q  \F_\xi^{-1} \bigg [ (\textbf{1}_{Q_q}(\xi) - \textbf{1}_{Q_q^\prime}(\xi)) \frac{\xi_\ast^\perp}{| \xi_\ast^\perp|} \bigg ] (x)  
\,\,\,\,
\text{for}
\,\,\,\,
x \in \R^3,  
\end{align}
where 
$\xi_\ast := (0, \xi_2, \xi_3)$, 
$\xi_\ast^\perp := (0,-\xi_3,\xi_2)$, and $C^\prime$ is an absolute constant determined later. 
\begin{proof}
We have  
\begin{align}
\F_\xi^{-1} \bigg [ (\textbf{1}_{Q_q}(\xi) - \textbf{1}_{Q_q^\prime}(\xi)) \frac{\xi_\ast^\perp}{|\xi_\ast^\perp|} \bigg ] (x) 
= 
2 \sin (\lambda_q x_1)
\bigg ( \frac{\sigma}{8} \bigg )^3
i \F^{-1}_\xi \bigg [\textbf{1}_Q (\xi) \frac{\xi_\ast^\perp}{| \xi_\ast^\perp|} \bigg ] \bigg ( \frac{\sigma x}{8} \bigg )
\,\,\,\,
\text{for}
\,\,\,\,
x \in \R^3, 
\end{align}
where $Q := \{ \xi = (\xi_1, \xi_2, \xi_3); \,\,|\xi_k| \leqq 1, \,\, k=1,2,3 \}$. 
Hence, for $q_0 \leqq q \leqq q_1$, we have 
\begin{equation} \label{eq:B_norm_estimate_example}
\lambda_q^{-1} | a_q (x) | 
\leqq  
\frac{\varepsilon}{8^3C^\prime }
\bigg \| \F^{-1}_\xi \bigg [\textbf{1}_Q (\xi) \frac{\xi_\ast^\perp}{| \xi_\ast^\perp|} \bigg ] \bigg \|_{L_x^\infty(\R^3)} 
< \varepsilon,  \qquad \forall x,
\end{equation}
where $C^\prime$ is given by 
\begin{align}
C^\prime = \bigg \| \F^{-1}_\xi \bigg [\textbf{1}_Q (\xi) \frac{\xi_\ast^\perp}{| \xi_\ast^\perp|} \bigg ] (x) \bigg \|_{L_x^\infty(\R^3)} 
\leqq \frac{1}{(2 \pi)^\frac{3}{2}} \bigg \| \textbf{1}_Q(\xi) \frac{\xi_\ast^\perp}{| \xi_\ast^\perp|} 
\bigg \|_{L^1_{\xi}(\R^3)} < \infty.
\end{align}
For $q < q_0-1$ or $q>q_1+1$, it holds that $|a_q(x)| = 0$ for all $x \in \R^3$. 
In particular, $a_{>q_1+1} = 0$. 
We then have $\| a \|_{\dot B^{-1}_{\infty,\infty}(\R^3)} < \varepsilon$. 
On the other hand, recalling that $\sigma = \lambda_{q_1}$, we note that 
\begin{align}
\|a_{\leqq q_0}\|_{BMO^{-1}} 
&= 
\bigg \{ \sup_{Q_{z,j} \in \mathcal{Q}} \frac{1}{|Q_{z,j}|} \int_{Q_{z,j}} \sum_{q = j}^{q_1} (\lambda_q^{-1} | a_q(x) |)^2 \, dx \bigg \}^\frac{1}{2} \\
&\geqq 
\frac{\varepsilon}{8^3 C^\prime} 
\bigg \{ \frac{1}{|Q_{0,q_0}|} \int_{Q_{0,-N}} \sum_{q = q_0}^{q_1} 
\bigg | \sin (\lambda_q x_1) i \F_\xi^{-1} \bigg [ \textbf{1}_{Q}(\xi) \frac{\xi_\ast^\perp}{| \xi_\ast^\perp|} \bigg ] \bigg ( \frac{\sigma x}{8} \bigg )  \bigg |^2 \, dx \bigg \}^\frac{1}{2} \\
&= 
\frac{\varepsilon }{8^{3/2} C^\prime} 
\bigg \{ 
\sum_{q = q_0}^{q_1} 
\int_{(0,1/8)^3} 
\bigg | 
\sin (8 \lambda_{q} x_1)
i \F_\xi^{-1} \bigg [ \textbf{1}_{Q}(\xi) \frac{\xi_\ast^\perp}{| \xi_\ast^\perp|} \bigg ] (x) \bigg |^2 \, dx \bigg \}^\frac{1}{2}. 
\end{align}
Now, using $\xi_* = (0,\xi_2,\xi_3)$, $x_* = (0,x_2,x_3)$, and $Q_* = \{ \xi_*; \,\, |\xi_k| \leqq 1, \,\, k=2,3 \}$, and noting  
\begin{align}
\sin (8 \lambda_{q} x_1) 
\F_\xi^{-1} \bigg [ \textbf{1}_{Q}(\xi) \frac{\xi_\ast^\perp}{| \xi_\ast^\perp|} \bigg ] (x)
=
\frac{\sin (8 \lambda_{q} x_1)}{(2 \pi)^{1/2}} 
\frac{2 \sin (x_1)}{x_1}
\F_{\xi_*}^{-1} \bigg [ \textbf{1}_{Q_*}(\xi_*) \frac{\xi_\ast^\perp}{| \xi_\ast^\perp|} \bigg] (x_*), 
\end{align}
we get 
\begin{align}
\| &a_{\leqq q_0} \|_{BMO^{-1}} \\
&\geqq 
\frac{\varepsilon }{16 \sqrt{\pi} C^\prime} 
\bigg \| \F_{\xi_*}^{-1} \bigg [ \textbf{1}_{Q_*}(\xi_*) \frac{\xi_\ast^\perp}{| \xi_\ast^\perp|} \bigg] (x_*) \bigg \|_{L^2_{x_*} ((0,\frac{1}{8})^2)}
\bigg \{
\sum_{q = q_0}^{q_1} 
\int^{\frac{1}{8}}_0 
\bigg |  
\frac{\sin (8 \lambda_{q} x_1) \sin (x_1)}{x_1} \bigg |^2 
\, d x_1 \bigg \}^\frac{1}{2}. \\
\end{align}
Since we have
\begin{align}
\sum_{q=q_0}^{q_1} 
\int_{0}^{\frac{1}{8}} \bigg | \frac{\sin(8 \lambda_{q} x_1) \sin(x_1)}{x_1} \bigg |^2 \, dx_1 & \gtrsim 
\sum_{q=q_0}^{q_1}  \int_{0}^{\frac{1}{8}} | \sin(8 \lambda_{q} x_1) |^2 \, dx_1\\
 & \gtrsim  \sum_{q=q_0}^{q_1} 1\\
&= q_1-q_0, 
\end{align}
it follows that
\[
\| a_{\leqq q_0} \|_{BMO^{-1}} > M,
\]
provided $q_1-q_0 \geq  \frac{C_1M}{\varepsilon}$,
for some absolute constant $C_1$.

Finally, we estimate the $L^2$ norm of $a$.
\[
\begin{split}
\|a\|_{L^2}^2 &= \left(\frac{\varepsilon}{2 C^\prime \sigma^3}\right)^2 \sum_{q=q_0}^{q_1} \lambda_q^2  \left\| (\textbf{1}_{Q_q}(\xi) - \textbf{1}_{Q_q^\prime}(\xi)) \frac{\xi_\ast^\perp}{| \xi_\ast^\perp|} \right\|_{L^2}^2\\
&\sim \left(\frac{\varepsilon}{2 C^\prime \sigma^3}\right)^2 \lambda_{q_1}^2 \sigma^3\\
&\sim \varepsilon^2 \lambda_{q_0}^{-3} \lambda_{q_1}^2\\
&=\varepsilon^2 \lambda_{q_1-q_0}^{3}\lambda_{q_1}^{-1}.
\end{split}
\]
Recalling that $\lambda_{q_1+1}\leqq 1/\sqrt{\delta}$, when $q_1-q_0 = \frac{C_1M}{\varepsilon}$, we have
\begin{equation}
\begin{split}
\|a\|_{L^2}^2 &\sim \varepsilon^2 2^{\frac{3C_1M}{\varepsilon}} \lambda_{q_1}^{-1}\\
&\gtrsim \varepsilon^2 2^{\frac{3C_1M}{\varepsilon}} \mu_0^{-1} E^{\frac{1}{2}} e^{-\frac{1}{2C_0 \varepsilon}}.
\end{split}
\end{equation}
In order to have $\|a\|_{L^2}^2 = E$, the following condition has to be satisfied
\[
E^{\frac{1}{2}} \gtrsim \varepsilon^2 2^{\frac{3C_1M}{\varepsilon}} \mu_0^{-1}  e^{-\frac{1}{2C_0 \varepsilon}},
\]
which concludes the proof.

\end{proof}

\section{Appendix}
For the equivalence of mild and weak formulation of the NSE in $L^\infty_tL^2_x$, we refer to 
\begin{Theorem}[Fabes, Jones and Riviere \cite{FabesJonesRiviere}] 
\label{FJR}
Assume $u \in L^q((0,T_*);L^p(\R^3))$ with $p,q\geqq 2$ and $p<\infty$. Then, $u$ is a mild solution of \eqref{NS} is and only if it is a weak solution. 
\end{Theorem}

\par \noindent 
{\bf Disclosure statement}: 
The authors declare that there are no competing interests to declare. 
\par \noindent
{\bf Data Availability Statements}: 
The authors confirm that the data supporting the findings of this study are available within the article.

\bibliography{WP_NS_BIB}

\end{document}